\documentclass[10pt,a4paper]{article}

\usepackage{amsmath,amsxtra,amssymb,latexsym,amscd,amsthm,array,multirow,makecell} 

\usepackage[mathcal]{euscript}

\usepackage{mathrsfs}

\usepackage{ifthen}

\usepackage[utf8]{inputenc}

\usepackage[T1]{fontenc}

\usepackage[vietnam,french,english]{babel}

\usepackage{indentfirst}

\usepackage{titlesec}

\usepackage[left=28.00mm, right=18.00mm, top=20.00mm]{geometry}

\numberwithin{equation}{section}

\usepackage{lmodern}

\allowdisplaybreaks

\usepackage{hyperref}

\def\ps{\mathscr{P}}

\def\F{\mathcal F}
\def\f2{\mathbb F_2}
\def\fp{\mathbb F_p}
\def\a2{\mathcal A_2}
\def\ap{\mathscr A}
\def\v{\mathcal V}
\def\vf{{\mathcal V}^f}

\def\n{\mathbb N}

\def\nil{\mathcal{N}il}

\def\u{\mathscr U}

\def\p{\mathcal P}

\newcommand{\ten}[3]{\overset{#2}{\underset{#1}{\bigotimes}}#3}

\newcommand{\ex}[4]{\mathrm{Ext}_{#1}^{#2}\left(#3,#4\right)}

\newcommand{\ho}[3]{\mathrm{Hom}_{#1}\left(#2,#3\right)}

\newcommand{\dsum}[3]{\underset{#1}{\overset{#2}{\bigoplus}}#3}
\newcommand{\som}[3]{\underset{#1}{\overset{#2}{\sum}}#3}

\makeatletter
\newcommand{\keywords}[1]{%
  \let\@@oldtitle\@title%
  \gdef\@title{\@@oldtitle\footnotetext{\emph{Key words and phrases.} #1.}}%
}
\makeatother

\newtheorem*{theorem*}{Theorem}
\newtheorem*{property*}{Property}
\newtheorem*{conj*}{Conjecture}
\newtheorem{thm}{Theorem}[section]
\newtheorem{pro}[thm]{Proposition}

\newtheorem{cor}[thm]{Corollary}
\newtheorem{lem}[thm]{Lemma}

\newtheorem{notat}[thm]{Notation}
\newtheorem{defi}[thm]{Definition}

\newcounter{nmdthmcnt}

    {\cleardoublepage\thispagestyle{plain}\null\begin{flushleft}%
    \Huge\bfseries Acknowledgements\end{flushleft}\vspace{1cm}\hspace{0.45cm}}%

\title{Homogeneous strict polynomial functors as unstable modules}
\author{{\selectlanguage{vietnam}NGUYỄN Thế Cường}\thanks{\selectlanguage{english}
I would like to thank Professor Lionel SCHWARTZ for urging me to prepare this paper. It is also a great pleasure for me to express my sincere thanks to all members of the VIASM for their hospitality during my two-month visit in Hanoi where the paper was baptised. I would like to take this opportunity to thank NGUYEN Dang Ho Hai and PHAM Van Tuan for having offered many significative discussions. This work would not be presented properly without precious suggestions from Vincent Franjou.}\hspace{1.5mm}\footnote{
Partially supported by the program ARCUS Vietnam MAE, Région IDF.
} }
\keywords{Steenrod algebra, strict polynomial functors, unstable modules}
\date{21st July 2014}

\AtEndDocument{\bigskip{\footnotesize%
  \begin{flushleft}
  \textsc{\selectlanguage{french}Départment de Mathématiques\hfill LIAFV - CNRS\\
      LAGA - Université Paris 13\hfill Formath Vietnam\\
      99 Avenue Jean-Baptiste Clément\\
      93430 Villetaneuse} \par  
    \textit{\bfseries E-mail address}, T. C.~NGUYEN: \texttt{tdntcuong@gmail.com} or  \texttt{nguyentc@math.univ-paris13.fr}\par
    \addvspace{\medskipamount}
  \end{flushleft}}}
       
\begin{document}
\maketitle

\begin{abstract}
A relation between Schur algebras and Steenrod algebra is shown in \cite{Hai} where to each strict polynomial functor the author associates an unstable module. We show that the restriction of Hai's functor to the subcategory of strict polynomial functors of a given degree is fully faithfull.
\end{abstract}


\section{Introduction}
The search for a relation between Schur algebras and the Steenrod algebra has been a source of common interest between representation theorists and algebraic topologists for over thirty years. 
Functorial points of view, on unstable modules \cite{HLS93}, and on modules over Schur algebras \cite{FS97}, have given an efficient setting for studying such relation.
For the Steenrod algebra side, \cite{HLS93} uses Lannes' theory to construct a functor $f$, from the category $\u$ of unstable modules, to the category $\F$ of functors from finite dimensional $\fp-$vector spaces to $\fp-$vector spaces.
This functor $f$ induces an equivalence between the quotient category $\u/\nil$ of $\u$ by the Serre class of nilpotent modules, and the full subcategory $\F_{\omega}$ of analytic functors. 
The interpretation of modules over Schur algebras given by \cite{FS97} uses an algebraic version of the category of functors, the category $\p$ of strict polynomial functors.
The category $\p$ decomposes as a direct sum  $\bigoplus_{d\geq 0}\p_{d}$ of its subcategories of
homogeneous functors of degree $d$. The category $\p_{d}$ is equivalent to the category of modules
over the Schur algebra $S(n,d)$ for $n \geq d$ \cite[Theorem 3.2]{FS97}. The presentation of
$\p_{d}$ as a category of functors  with an extra structure comes with a functor $\p_{d}\to \F$. 
Nguyen D. H. Hai showed  \cite{Hai} that this functor  $\p_{d}\to \F$
factors through the category $\u$ by a functor $\bar{m}_{d}:\p_{d}\to\u$. These functors $\bar{m}_{d}$ induce a functor $\bar{m}:\p\to\u$. The functor $\bar{m}$ has remarkably interesting properties. In particular, it is exact and commutes with tensor products. The relevance of this last property to computation will soon be apparent.

We observe that Hom-groups between unstable modules coming from strict polynomial functors via Hai's functor are computable: they are isomorphic to the Hom-groups of the corresponding strict polynomial functors in many interesting cases. 
The primary goal of this paper is to generalize these results to the whole category $\p_{d}$. The main theorem of the present work goes as follows:
\begin{thm}\label{main1}
The functor $\bar m_d:\p_{d}\to \u$ is fully faithful.
\end{thm}
The theorem is proved by comparing corresponding Hom-groups in the two categories. We discuss an example of interest. Let $n$ be an integer and $V$ be an $\f2-$vector space. The symmetric group $\mathfrak{S}_{n}$ acts on $V^{\otimes n}$ by permutations. Denote by $\Gamma^{n}(V)$ the group of invariants  $\left(V^{\otimes n}\right)^{\mathfrak{S}_{n}}$ and by $S^{n}$ the group of coinvariants  $\left(V^{\otimes n}\right)_{\mathfrak{S}_{n}}$. The free unstable module generated by an element $u$ of degree $1$ is denoted by $F(1)$. It has an $\f2-$basis consisting of $u^{2^{k}}$ with $0\leq k$. To a strict polynomial functor $G$, Hai's functor $\bar{m}$ associates the unstable module $G(F(1))$. 

Denote by $\mathrm{Sq}_{0}$ the operation which associates to an homogenous element $x\in M^{n}$ of
$M\in\u$ the element $Sq^{n}x$. An unstable module $M$ is nilpotent if for every $x\in M^{n}$ there
exist an integer $N_{x}$ such that $\mathrm{Sq}_{0}^{N_{x}}x=0$. An unstable module $M$ is reduced
if
$\ho{\u}{N}{M}$ is trivial for every nilpotent module $N$. It is called $\nil-$closed if
$\ex{\u}{i}{N}{M},i=0,1$ are trivial for every nilpotent module $N$. 

We now show that
$$\ho{\p_{3}}{\Gamma^{2}\otimes \Gamma^{1}}{S^{3}}\cong\ho{\u}{\Gamma^{2}(F(1))\otimes \Gamma^{1}(F(1))}{S^{3}(F(1))}.$$ 
The readers of \cite{HLS93} might expect the latter Hom-group to be isomorphic to
$\ho{\F}{\Gamma^{2}\otimes \Gamma^{1}}{S^{3}}$. However $S^{3}(F(1))$ is not $\nil-$closed
\footnote{ The submodule of $S^{3}(F(1))$ generated by $u.u.u^{4}$ is concentrated in even degrees
but this element does not has a square root.} then such an expectation fails. By classical functor
techniques \cite[Theorem 1.7]{FFSS99}, if $\mathcal{C}$ is $\p$ or $\F$ then:
$$\ho{\mathcal{C}}{\Gamma^{2}\otimes \Gamma^{1}}{S^{3}}\cong \bigoplus_{i=0}^{3}\ho{\mathcal{C}}{\Gamma^{2}}{S^{i}}\otimes \ho{\mathcal{C}}{\Gamma^{1}}{S^{3-i}}$$
It follows that: 
\begin{align*}
\ho{\F}{\Gamma^{2}\otimes \Gamma^{1}}{S^{3}}&\cong \f2^{\oplus 2},\\
\ho{\p_{3}}{\Gamma^{2}\otimes \Gamma^{1}}{S^{3}}&\cong \f2.
\end{align*} 
The module $S^{3}(F(1))$ is not $\nil-$closed but it is reduced. On the other hand, the quotient of
the module $\Gamma^{2}(F(1))\otimes \Gamma^{1}(F(1))$ by its submodule generated by $u\otimes
u\otimes u^{4}$ is nilpotent. Therefore:
\begin{align*}
\ho{\u}{\Gamma^{(2,1)}(F(1))}{S^{3}(F(1))}&\cong \ho{\u}{\left\langle  u\otimes u\otimes u^{4} \right\rangle }{S^{3}(F(1))}\\
&\cong \f2.
\end{align*}
This example is the key to the proof of Theorem \ref{main1}.

We end the introduction by giving some further remarks on the result and stating the organization of the paper.

Theorem \ref{main1} implies that the category $\p_{d}$ is a full subcategory of the category $\u$. Unfortunately the category $\p$ itself cannot be embedded into $\u$. There is no non-trivial morphism in the category $\p$ from $\Gamma^{2}$ to $\Gamma^{1}$ but
\begin{align*}
\ho{\u}{\bar{m}(\Gamma^{2})}{\bar{m}(\Gamma^{1})}&\cong \ho{\u}{F(2)}{F(1)}\\
&\cong \fp.
\end{align*}

The category $\p_d$ is not a thick subcategory of  $\u$. Fix $p=2$, let $I^{(1)}$ denote the
Frobenius twist in $\p_{2}$, that is the base change along the Frobenius. It is proved in
\cite{FS97} that
$$
\ex{\p_2}{i}{I^{(1)}}{I^{(1)}}\cong\left\{\begin{array}{ll}
\f2&\textup{ if }i=0,1,\\
0&\text{ otherwise}.
\end{array}\right.
$$
On the other hands, the corresponding Ext group in the category $\u$ is $\ex{\u}{i}{\Phi F(1)}{\Phi
F(1)}$. As in \cite{Cuo14b}:
$$
\ex{\u}{i}{\Phi F(1)}{\Phi F(1)}\cong\left\{\begin{array}{ll}
\f2&\textup{ if }i=2^{n}-1,\\
0&\text{ otherwise}.
\end{array}\right.
$$ Therefore, $\ex{\p_2}{i}{I^{(1)}}{I^{(1)}}$ is not isomorphic to $\ex{\u}{i}{\bar{m}_{d}(I^{(1)})}{\bar{m}_{d}(I^{(1)})}\text{ for }i=2^{n}-2, n\geq 3$.
\subsection*{Organization of the article}

In section 2 we recall basic facts on the Steenrod algebra and unstable modules in the sense of
\cite{Hai}. When $p>2$, our version of the Steenrod algebra is slightly different from the version
in the sense of \cite{Ste62} since we do not consider the Bockstein operation. We also recall
Milnor's coaction on unstable modules and how it is used to determine action of the Steenrod
algebra on certain type of elements. 

The next section recalls strict polynomial
functors. The construction of Hai's functor is introduced and an easy observation on the existence
of its adjoint functors is also
given.

The structure of $\Gamma^{\lambda}(F(1))$ is treated in section 4. We show that
this module is \textit{monogenous modulo nilpotent}. A special class of \textit{generators modulo
nilpotent} of this module is computed. 

The last section deals with Theorem \ref{main1}. The proof of this theorem is based on a
combinatorial process followed by some Steenrod algebra techniques. 
\section{Steenrod algebra and unstable modules}
In this section, we follow the simple presentation in \cite[section 3]{Hai} to define the Steenrod
algebra and unstable modules. 

The letter $p$ denotes a prime number. Let $[-]$ be the integral part of a number. We denote by
$\ap$ the quotient of the free associative unital graded $\fp-$algebra generated by the $\ps^{k}$ of
degree $k(p-1)$ subject to the Adem relations 
$$\ps^{i}\ps^{j}=\sum_{t=0}^{\left[\frac{i}{p}\right]}\binom{(p-1)(j-t)-1}{i-pt}\ps^{i+j-t}\ps^{t}$$
for every $i\leq pj$ and $\ps^{0}=1$ \cite[section 3]{Hai}. 

An $\ap-$module $M$ is called unstable
if $\ps^{k}x$ is trivial as soon as $k$ is strictly greater than the degree of $x$. We denote by
$\u$ the category of unstable modules. 

Let $\mathcal{A}_{p}$ be the Steenrod algebra \cite{Ste62,Sch94}. If $p=2$ then there is an
isomorphism of algebras $\ap\to \a2$, obtained by identifying the $\ps^{k}$ with the Steenrod
squares $Sq^{k}$. The category $\u$ is equivalent to the category $\mathcal{U}$ of unstable modules
in the sense of \cite{Sch94}. If $p>2$, $\ap$ is isomorphic, up to a grading scale, to the
sub-algebra of $\mathcal{A}_{p}$
generated by the reduced Steenrod powers $P^{k}$ of degree $2k(p-1)$. The category $\u$ is
equivalent
to the subcategory $\mathcal{U}'$ of unstable $\mathcal{A}_{p}-$modules concentrating in even
degrees \cite[section 1.6]{Sch94}. 

By abuse of terminology, we call $\ap$ the Steenrod algebra and
$\ps^{k}$ the $k-$th reduced Steenrod power.

Serre \cite{Ser53} introduced the notions of \textit{admissible} and \textit{excess}. The monomial 
$$\ps^{i_{1}}\ps^{i_{2}}\ldots \ps^{i_{k}}$$
is called admissible if $i_{j}\geq pi_{j+1}$ for every $1\leq j\leq k-1$ and $i_{k}\geq 1$. The
excess of this operation, denoted by $e\left( \ps^{i_{1}}\ps^{i_{2}}\ldots \ps^{i_{k}}\right)$, is
defined by $$e\left( \ps^{i_{1}}\ps^{i_{2}}\ldots \ps^{i_{k}}\right)=
pi_{1}-(p-1)\left(\sum_{j=1}^{k}i_{j}\right).$$ The set
of admissible monomials and $\ps^{0}$ is an additive basis of $\ap$.

Let $|-|$ be the degree of an element. Denote by $\mathrm{P}_{0}$ the operation which associates to
an homogenous element $x\in M^{n}$ of
$M\in\u$ the element $\ps^{|x|}x$.
An unstable module $M$ is nilpotent if for every $x\in M^{n}$ there exist an integer $N_{x}$ such
that $\mathrm{P}_{0}^{N_{x}}x=0$. Denote by $\nil$ the class of all nilpotent modules. 

An unstable module $M$ is reduced if $\ho{\u}{N}{M}$ is trivial for every nilpotent module $N$. It
is called $\nil-$closed if $\ex{\u}{i}{N}{M},i=0,1$ are trivial for every nilpotent module $N$. 

\begin{defi}[Mod-nil epimorphism]
A morphism of unstable module $M\to N$ is called mod-nil epimorphism if its cokernel is nilpotent  .
\end{defi}

Let $n$ be a positive integer. We denote by $F(n)$ the free unstable module generated by a generator
$\imath_{n}$ of degree $n$. These $F(n)$ are projective satisfying
$\ho{\u}{F(n)}{M}\cong M^{n}$. When $n=1$ such a generator is denoted by $u$ rather than
$\imath_{1}$. As an $\fp-$vector space, $F(1)$ is generated by $u^{p^{i}},i\geq 0$. Action of the
reduced Steenrod power $\ps^{k}$ is defined by:
$$\ps^{k}\left( u^{p^{i}}\right)=\left\{ \begin{array}{ll}
                                          u^{p^{i}}&\textup{ if }k=0,\\
                                          u^{p^{i+1}}&\textup{ if }k=p^{i},\\
                                          0 & \textup{ otherwise}.
                                         \end{array}
  \right.$$
There is an isomorphism of unstable modules $F(n)\cong \left(F(1)^{\otimes
n}\right)^{\mathfrak{S}_{n}}$ where the symmetric group $\mathfrak{S}_{n}$ acts by permutations.
Then we can identify $F(n)$ with the submodule of $F(1)^{\otimes n}$
generated by
$u^{\otimes n}$ \cite[section 1.6]{Sch94}. 
\begin{defi}[Mod-nil generator]	
An unstable module $M$ is \textit{mod-nil monogeneous} if there exists a mod-nil epimorphism
$f:F(n)\to M$. The element $f(\imath_{n})$ is called a \textit{mod-nil generator} of $M$.
\end{defi}

Milnor \cite{Mil58} established that $\ap$ has a natural coproduct which makes it into
a Hopf algebra and incorporates Thom's involution as the conjugation. The dual $\ap^{*}$ of $\ap$ is
isomorphic to the polynomial algebra 
$$\fp[\xi_{0},\xi_{1},\ldots,\xi_{k},\ldots],\hspace{1mm}|\xi_i|=p^{i}-1,\xi_0=1.$$ 

Let $R=(r_{1},r_{2},\ldots,r_{k},\ldots)$ be a sequence of integers with  only finitely many
non-trivial ones. Denote by $\xi^{R}$ the product $\xi_{1}^{r_{1}}\xi_{2}^{r_{2}}\ldots
\xi_{k}^{r_{k}}\ldots$ These monomials form a basis for $\ap^{*}$. Let $m_{n}(r)\in\ap$
denote the dual of $\xi_r^n$ with respect to that basis.

If $M$ is an unstable module, the completed tensor product $M\hat{\bigotimes}\ap^{*}$ 
is the $\fp-$graded vector space defined by:
$$\left(M\hat{\bigotimes}\ap^{*}\right)^{n}=\prod_{l-k=n}^{}M^{l}\otimes
\left(\ap^{*}\right)^{k}.$$

We recall how to use Milnor's coaction to determine the action of the Steenrod algebra. There is
Milnor's coaction $\lambda:M\to M\hat{\bigotimes}\ap^{*}$ for an unstable module $M$. We write
$\lambda(x)$ as a formal sum $\sum_{R}^{}x_{R}\otimes \xi^{R}$. Let $\theta$ be a Steenrod operation
then 
$$\theta x=\sum_{R}^{}x_{R}\otimes \xi^{R}(\theta).$$
Milnor's coaction on a tensor product is determined as follows:
$$\lambda (x\otimes y)=\sum_{R}^{}\sum_{I+J=R}^{}(x_{I}\otimes y_{J})\otimes\xi^{R}.$$
Milnor's coaction on $F(1)$ is defined by:
$$\lambda(u)=\sum_{i\geq 0}^{}u^{p^{i}}\otimes\xi_{i} \textup{ and }\lambda(u^{p^{s}})=\sum_{i\geq
0}^{}u^{p^{s+i}}\otimes\xi^{p^{s}}_{i}.$$

The following observation is easy and is left to the reader.
\begin{lem}\label{mil1}
Let $n$ be an integer then $m_n(r)\left(u^{\otimes n}\right)=\left(u^{p^r}\right)^{\otimes n}$. If
$k_1,\ldots,k_{m}$ is a sequence of integer such that $p^{k_{j}}>n$ for every $j$ then:
\begin{align*}
m_n(r)\left(u^{\otimes n}\otimes u^{p^{k_{1}}}\otimes
u^{p^{k_{2}}}\otimes\cdots\otimes u^{p^{k_{m}}}\right)&=\left(u^{p^r}\right)^{\otimes n}\otimes
\left(u^{p^{k_{1}}}\otimes u^{p^{k_{2}}}\otimes\cdots\otimes u^{p^{k_{m}}}\right).
\end{align*}
\end{lem}
The following proposition is a strengthening of Lemma \ref{mil1}:
\begin{pro}\label{Stee} Let $M$ and $N$ be two unstable modules.
For all class $x \in M$ of degree $n$ we have $m_n(r)(x)=\mathrm{P}_0^r(x)$.
For all class $y \in M$ and all natural number $k $ such that $p^k>n$ we have
$$
m_n(r) (x \otimes \mathrm{P}_0^k(y)) =\mathrm{P}_0^r (x) \otimes \mathrm{P}_0^k (y).
$$
\end{pro}
\begin{proof}
Consider the morphism $\varphi:F(n)\to M$ defined by $\imath_n=u^{\otimes n}\mapsto x$. Lemma
\ref{mil1} yield:
\begin{align*}
m_n(r)(x)&=m_n(r)\left(\varphi\left(u^{\otimes n}\right)\right)\\
&=\varphi\left(m_n(r)\left(u^{\otimes n}\right)\right)\\
&=\varphi\left(\mathrm{P}_0^r\left(u^{\otimes n}\right)\right)\\
&=\mathrm{P}_0^r\left(\varphi\left(u^{\otimes n}\right)\right)\\
&=\mathrm{P}_0^r(x).
\end{align*}
Similarly, by considering the morphism $\psi:\Phi^{k}(F(|y|))\to N$ defined by
$\ps^{p^{k}}\imath_{|y|}\mapsto \ps^{p^{k}} y$ and the product $\varphi\otimes\psi$, we obtain the
second equality.
\end{proof}

\section{Strict polynomial functors and Hai's functor}\label{section22}
The main goal of this section is to recall Hai's functor and give an easy observation on the
existence of its adjoint functors in this section. 

Following the simple presentation introduced in \cite{Pir03} we first recall the category of strict polynomial functors. Fix a prime number $p$, denote by $\v$ the category of $\fp-$vector spaces and by $\v^{f}$ its full subcategory of spaces of finite dimension. Let $n$ be a positive integer. Denote by $\Gamma^{n}(V)$ the group of invariants $\left(V^{\otimes n}\right)^{\mathfrak{S}_{n}}$. The category $\Gamma^{d}\v^{f}$ is defined by:
\begin{align*}
\textup{Ob}(\Gamma^d\vf)&=\textup{Ob}(\vf),\\ 
\ho{\Gamma^d\vf}{V}{W}&=\Gamma^d(\ho{\vf}{V}{W}).
\end{align*}
A homogeneous strict polynomial functor of degree $d$ is an $\fp-$linear functor from $\Gamma^d\vf$ to $\vf$. We denote by $\p_{d}$ the category of all these functors. The notation $\p$ stands for the direct sum $\dsum{d\geq 0}{}{\p_d}$. A strict polynomial functor is an object of the category $\p$.

We now recall the parametrized version of $\Gamma^{d}$ and $S^{d}$. For each $W\in\vf$, let $\Gamma^{d,W}$ be the functor which associates to an $\fp-$vector space  $V$ the $\fp-$vector space $\Gamma^d(\ho{\v^{f}}{W}{V})$, and let $S^{d,W}$ be the functor which associates to  an $\fp-$vector space  $V$  the $\fp-$vector space $S^d(W^{\sharp}\otimes V)$. Here, $W^{\sharp}$ stands for the linear dual of $W$. The $\Gamma^{d,W}$ are projective satisfying $\ho{\p_d}{\Gamma^{d,W}}{F}\cong F(W)$ and the $S^{d,W}$ are injective satisfying $\ho{\p_d}{F}{S^{d,W}}\cong F(W)^{\sharp}$. 

Let $\lambda=\left(\lambda_{1},\ldots,\lambda_{k}\right)$ be a sequence of integers. Denote by $\Gamma^{\lambda}$ the tensor product $\Gamma^{\lambda_{1}}\otimes \cdots\otimes \Gamma^{\lambda_{k}}$. Hai's functor $\bar{m}_d\colon \p_d\rightarrow \u$ is defined as follows. Given a strict polynomial functor $F$ in $\p_d$. In degree $e$, $\bar{m}_d(F)$ is defined to be the vector space of natural transformations from $\bar{\Gamma}^{d;e}$ to $F$:
$$\bar{\Gamma}^{d;e}:=\bigoplus\limits_{\substack{|\lambda|=d\\||\lambda||=e}}\Gamma^\lambda,\quad
\bar{m}_d(F)^e:= \ho{\p_d}{\bar{\Gamma}^{d;e}}{F},$$ where $|\lambda|,||\lambda||$ stand for
$\lambda_1 +\ldots +\lambda_n$ and $\lambda_1+p \lambda_2+ \ldots +p^{n-1}\lambda_n$ respectively.
The structural morphisms are induced by  $\ho{\p_d}{\bar{\Gamma}^{d;e}}{-}$. In other words, the
functor $\bar{m}_{d}$ can be defined as the evaluation on $F(1)$:
$$\bar{m}_{d}(G)\cong G(F(1))\textup{ for all }G\in \p_{d},$$
and the structural morphisms are defined by:
\begin{align*}
\left(\bar{m}_{d}\right)_{V,W}:\ho{\p_d}{F_1}{F_2}&\to \ho{\u}{\bar{m}_{d}(F_1)}{\bar{m}_{d}(F_2)}\cong \ho{\u}{F_1(F(1))}{F_2(F(1))}\\
f&\mapsto f(F(1)).
\end{align*}
Let $\bar{m}$ denote the induced functor from $\p$ to $\u$. 
The functor $\bar{m}$ has nice properties. It is exact and commutes with tensor products as well as Frobenius twists \cite[see sections 3 and 4]{Hai}. 

The following observation is easy and is left to the reader:
\begin{pro}\label{exist}
The functor $\bar{m}_{d}$ admits a left adjoint denoted by $l_d$ and a right adjoint denoted by $r_d$.
\end{pro} 

\section{Mod-nil generators of some unstable modules}
As explained in the introduction, we prove Theorem \ref{main1} by comparing corresponding Hom-groups in the two categories $\p_{d}$ and $\u$. This computation can be reduced to a smaller class of strict polynomial functors. This class is described in the following proposition:
\begin{pro}[\cite{FS97}]
If $\mathrm{dim}_{\fp}W\geq d$ then $S^{d,W}$ is an injective generator of $\p_d$. The functors $
\Gamma^{\lambda}$, where $\lambda$ runs through the set of all sequences of integers whose sum is
$d$, form a system of projective generators of $\p_d$ \cite{FS97}. \end{pro}
Therefore Theorem \ref{main1} is equivalent to the following lemma.
\begin{lem}\label{key12}
There are isomorphisms:
$$\ho{\u}{\Gamma^{\lambda}(F(1))}{S^{d,V}(F(1))}\cong S^{\lambda_1}(V^{\sharp})\otimes\cdots\otimes S^{\lambda_k}(V^{\sharp})$$
for every $\lambda=(\lambda_{1},\lambda_{2},\ldots,\lambda_{k}), |\lambda|=d$.
\end{lem}
In this section we show that $\Gamma^{\lambda}(F(1))$ is mod-nil monogenous.
The following lemma show that tensor products
respect mod-nil monogenous modules.
\begin{lem}\label{modnilgeneral}
	If $M$ and $N$ are mod-nil monogeneous then so is their tensor product $M\otimes N$. 
\end{lem}
Since $\mathrm{P}_{0}(x\otimes y)=\mathrm{P}_{0}(x)\otimes \mathrm{P}_{0}(y)$ then tensor products
respect mod-nil epimorphisms. Therefore Lemma \ref{modnilgeneral} is a consequence of the following
lemma:
\begin{lem}\label{modnil57}
For all $n,m\in\n$ the module $F(n)\otimes F(m)$ is mod-nil monogeneous. More precisely, if $q$ is a
number such that $p^{q}>n$ then $\imath_{n}\otimes \mathrm{P}_{0}^{q}\imath_{m}$ is a mod-nil
generator of $F(n)\otimes F(m)$.
\end{lem}
\begin{proof}
Let $q$ be a number such that $p^q>n$ and denote $\imath_n\otimes \mathrm{P}_0^q \imath_m$ by
$\alpha$. We prove that $\alpha$ is a mod-nil generator of the tensor product $F(n)\otimes F(m)$. In
other words, for each $\gamma,\beta\in\ap$ we show that there exists a number
$N_{\gamma,\beta}$ such that $\mathrm{P}_0^{N_{\gamma,\beta}}(\gamma\imath_n\otimes\beta \imath_m)$
belongs to the submodule of $F(n)\otimes F(m)$ generated by $\alpha$.

The inequality $p^{q}>n$ allows to apply Lemma \ref{mil1}: $$m_{n}^{q}\left(\imath_n\otimes
\mathrm{P}_0^q
\left(\theta \imath_m\right)\right)=\mathrm{P}_0^q\left(\imath_n\otimes \theta\imath_m\right)$$
for every Steenrod operation $\theta$. Express $\theta$ as a sum of admissible monomials
$\ps^{i_{1}}\ldots \ps^{i_{k}}$ and denote by $\theta_{0}$ the sum of $\ps^{p^{q}i_{1}}\ldots
\ps^{p^{q}i_{k}}$. As $p^q>n$ then  $$m_{n}^{q}\left(\imath_n\otimes \mathrm{P}_0^q \left(\theta
\imath_m\right)\right)=m_{n}^{q}\theta_{0}\left(\alpha\right).$$ Hence
$\mathrm{P}_0^q\left(\imath_n\otimes \theta\imath_m\right)$ belongs to the submodule
$\ap\left(\alpha\right)$.

It remains to show that for each admissible monomial  $\theta$, there exists $N\in\n$ such that
$\mathrm{P}_0^N\left(\theta \imath_n\otimes \delta \imath_m\right)$ belongs to $\ap
	(\alpha)$ for all Steenrod operation $\delta$. We make an induction on the degree of
$\theta$. Since the case of degree $0$ is verified above we may proceed to the induction argument by
supposing that the statement holds for all admissible monomial $\theta$ of degree less than $k$. Let
$\ps^I$ be an admissible monomial of degree $k$. For technical reason, its first term $\ps^{l}$ is
written separately: $\ps^{I}=\ps^l\omega$. By the Cartan formula we have
	$$
	\ps^l\omega \imath_n\otimes\delta \imath_m
	=\ps^l(\omega \imath_n\otimes\delta \imath_m)
	-\sum_{i=0}^{l-1}\ps^{i}\omega \imath_n\otimes
	\ps^{(l-i)}\delta \imath_m.
	$$
	By induction hypothesis, there exists $N$ such that $\mathrm{P}_0^N(\ps^l\omega
	\imath_n\otimes\delta \imath_m)$ belongs to $\ap (\alpha)$ for all Steenrod operation 
	$\delta$. This concludes the lemma. 
\end{proof}
These two results yield:
\begin{cor}
For all sequence of natural numbers $\lambda=(\lambda_1,\ldots,\lambda_n)$, the module $F(\lambda)$
is mod-nil monogeneous. More precisely, if $\left\{\alpha_i,n\geq i\geq 2 \right\}$ are $n-1$
integers such that
$\alpha_{k}>\lambda_{1}+p^{\alpha_{2}}\lambda_{2}+\cdots+p^{\alpha_{k-1}}\lambda_{k-1}$ then the
element  $\imath_{\lambda_1}\otimes \mathrm{P}_0^{\alpha_2}\imath_{\lambda_2}\otimes\cdots\otimes
\mathrm{P}_0^{\alpha_{n}}\imath_{\lambda_n}$ is a mod-nil generator of $F(\lambda)$.
\end{cor}
The use of mod-nil generators goes back \cite{FS90}:
\begin{pro}
The element $u\otimes u^{p}\otimes\cdots\otimes u^{p^{n-1}}$ is a mod-nil generator of the module
$F(1)^{\otimes n}.$ The images of this element by the canonical projections $\bigotimes^n\to S^n$
and $\bigotimes^n\to \Lambda^n$ are therefore mod-nil generators of the modules $S^n(F(1))$ and
$\Lambda^n(F(1))$ respectively.
\end{pro}
As mentioned before in the introduction, $S^{d,V}(F(1))$ is not $\nil-$closed in general. However it is reduced. This property allows to show that each morphism $F(\lambda)\to S^{d,V}(F(1))$ is determined by the image of a mod-nil generator of $F(\lambda)$ in $S^{d,V}(F(1))$.

\begin{defi}
If $M\in\u$ is reduced and if  $\mathrm{P}_0^n(x)=z$, we define $\sqrt[p^{n}]{z}=x$.
\end{defi}
\begin{lem}
Let $\alpha$ be a mod-nil generator of an unstable mod-nil monogeneous  module  $M$. If $N$ is a  reduced  unstable module then all morphisms from $M$ to $N$ are determined by the image of  $\alpha$ in $N$.
\end{lem}
\begin{proof}
Let $f$ be a morphism from $M$ to $N$ and let $x$ be an arbitrary element of  $M$. There are a
natural number $n$ and a Steenrod operation $\theta$ such that  $\mathrm{P}_0^n(x)=\theta(\alpha)$.
It follows that 
\begin{align*}
\mathrm{P}_0^n(f(x))&=f(\mathrm{P}_0^n(x))\\
&=f(\theta(\alpha))\\
&=\theta f(\alpha).
\end{align*}
 As $N$ is reduced then $f(x)=\sqrt[p^n]{\theta f(\alpha)}$.
\end{proof}

We are thus led to the problem of determining the subspace of $S^{d,V}(F(1))$ of all possible images of a mod-nil generator of $F(\lambda)$. Lemma \ref{prin2} presents the desired determination. Before formulating this lemma, we fix the following notation:
\begin{notat}
Let $\alpha=\left(\alpha_1,\alpha_2,\ldots,\alpha_t\right)$, we denote:
$$\omega_{\alpha} = \mathrm{P}_0^{\alpha_1}\imath_{\lambda_1} \otimes \mathrm{P}_0^{\alpha_2}
\imath_{\lambda_2} \otimes \cdots \otimes \mathrm{P}_0^{\alpha_t} \imath_{\lambda_t}.$$
\end{notat}
Let us present an example of interest. There is an isomorphism of $\fp-$vector space:
$$\ho{\u}{F(d)}{S^{d,V}(F(1))}\cong\left(S^{d,V}(F(1))\right)^{d}.$$ Therefore if $\varphi$ is a
morphism from $F(d)$ to $S^{d,V}(F(1))$ then the image $\varphi(\imath_{d})$ is a sum of elements of
the type 
$\prod_{i=1}^{d}s_{i}\otimes u$ where $s_{i}\in V^{\sharp}$. The natural transformation
$\ten{j=1}{k}{S^{\lambda_j,V}}\to S^{d,V}$ induces a morphism 
$$\rho:\ten{j=1}{k}{\ho{\u}{F(\lambda_j)}{S^{\lambda_j,V}(F(1))}}\to\ho{\u}{F(\lambda)}{S^{d,V}
(F(1))}.$$ For $1\leq i\leq t$, let $f_{i}$ be a morphism from $F(\lambda_i)$ to
$S^{\lambda_i,V}(F(1))$. Then the image of $\omega_{\alpha}$ by
$\rho\left(\bigotimes_{i=1}^{t}f_{i}\right)$ is a sum of elements of simple types 
$$\prod_{k=1}^{t}\left(\prod_{i=1}^{\lambda_{k}}s_{i_{k}}\otimes u^{p^{\alpha_{k}}}\right)$$
with  $s_{i_{k}}\in V^{\sharp}$. We show that every morphism in $\ho{\u}{F(\lambda)}{S^{d,V}(F(1))}$ is of this simple form.
\begin{lem}\label{prin2}
Let $\lambda=(\lambda_{1},\ldots,\lambda_{t})$ be a sequence of integers whose sum is $d$ and
$\delta$ be a number such that $\delta>\max\left\{\lambda_i\right\}$. Denote by $\alpha$ the
sequence $(0,\delta,2\delta,\ldots,(t-1)\delta)$. Then the element $\omega_\alpha$ is a mod-nil
generator of $F(\lambda)$. Moreover if $\varphi$ is a morphism from $F(\lambda)$ to $S^{d,V}(F(1))$
then $\varphi(\omega_\alpha )$ is a sum of elements of type
$$\prod_{k=0}^{t-1}\left(\prod_{i=1}^{\lambda_{k+1}}s_{i_{k}}\otimes u^{p^{k\delta}}\right)$$
where $s_{i_{k}}\in V^{\sharp}$.
\end{lem}
The morphism $\rho$ is clearly injective. Lemma \ref{prin2} show that this morphisms is surjective as well and hence it is bijective. Since $$\ten{j=1}{k}{\ho{\u}{F(\lambda_j)}{S^{\lambda_j}(F(1)\otimes V^{\sharp})}}\cong\ten{j=1}{k}{S^{\lambda_j}(V^{\sharp})},$$ Lemma \ref{key12} is a corollary of Lemma \ref{prin2}.
\section{Proof of the key lemma}\label{theproof}
As discussed in previous section, we are left with Lemma \ref{prin2}. To deal with this lemma we first investigate some combinatorial lemmas.

\subsection{A combinatorial process}\label{combpro}
A morphism $\varphi$ in $\ho{\u}{F(\lambda)}{S^{d,V}(F(1))}$ is a graded morphism. Then an easy observation on the degree of the image of a mod-nil generator of $F(\lambda)$ leads us to several combinatorial lemmas of this paragraph.
\begin{lem}\label{lp}
Let $n$ be a number equal to a sum of $p-$powers $\sum_{i=0}^{q}p^{l_i}$ where $l_0\leq
l_1\leq\ldots\leq l_q$. Let $\sum_{i=1}^{k}\lambda_{i}p^{m_i}$ be the $p-$decomposition of $n$
with $m_1<m_2<\ldots<m_k$. Then there exists a partition
$S_1,S_2,\ldots,S_k$ of $\left\{l_0,l_1,\ldots,l_q\right\}$  such that $\som{j\in
S_i}{}{p^{l_j}}=\lambda_{i}p^{m_i}$. Therefore $q+1\geq k$.
\end{lem}
\begin{proof}
We proceed by induction on $m_k$. The lemma is trivial for the case $m_k=1$. Suppose that the lemma is verified for all $m_k<n$ we prove it for the case $m_k=n$.
\begin{enumerate}
\item If $m_1,l_0>0$, then by dividing both sides of the equality 
$\som{i=1}{k}{\lambda_{i}p^{m_i}}=\som{i=1}{q}{p^{l_i}}$ by $p$ we return to the case $m_k=n-1$.
\item If $m_1=0$, let $a(q)$ be the index that
$$0=l_0=\ldots=l_{a(q)}<l_{a(q)+1}.$$
Because $n\equiv \lambda_{1}(\textup{ modulo }p)$ then
$a(q)-\lambda_{1}$ is divisible by $p$. Therefore
\begin{equation}
\label{lq}\som{i=2}{k}{\lambda_{i}p^{m_i}}=p\left(
p^{l_{\lambda_{1}+1}}+p^{l_{p\lambda_{1}+1}}+\ldots+p^{l_{a(q)-\lambda_{1}-p+1}}
\right)+\som{i=a(q)+1}{q}{p^{l_i}}.
\end{equation}
By dividing both side of
\ref{lq} by $p$ we have:
$$\som{i=2}{k}{\lambda_{i}p^{m_i-1}}=p^{l_{\lambda_{1}+1}}+p^{l_{p\lambda_{1}+1}}+\ldots+p^{l_{
a(q)-\lambda_{1}-p+1}}+\som{i=a(q)+1}{q}{p^{l_i-1}} .$$ According to induction hypothesis, we have a
partition of index sets
$$\left\{l_{\lambda_{1}+1},l_{p\lambda_{1}+1},\ldots,l_{a(q)-\lambda_{1}-p+1}\right\}\cup\left\{l_{
a(q)+1}-1, \ldots,l_{q} -1\right\}$$ into $k-1$ subsets $T^{}_2,\ldots,T^{}_{k}$ such that
$$\som{j\in T^{}_i}{}{p^{j}}=\lambda_{i}p^{m_i-1}.$$ For each $i\geq 2$ we denote:
\begin{align*}
S^{'}_i&=T_i\cap
\left\{l_{\lambda_{1}+1},l_{p\lambda_{1}+1},\ldots,l_{a(q)-\lambda_{1}-p+1}\right\},\\
T^{'}_i&=T_i\cap \left\{l_{a(q)+1}-1,\ldots,l_{q}-1\right\}.
\end{align*}
We write $S_1=\left\{l_0,l_{1},\ldots,l_{\lambda_{1}}\right\}$ and
\begin{align*}
S_i&=S^{'}_i\cup\left\{l_{pj\lambda_{1}+2},l_{pj\lambda_{1}+3},\ldots,l_{pj\lambda_{1}+p}\left|l_{
pj\lambda_{1}+1}\in S^{'}_i \right.\right\} \cup\left\{l+1\left|l\in T^{'}_i\right.\right\},2\leq
i\leq k,
\end{align*} 
then $S_1,S_2,\ldots,S_k$ is a partition for the case $m_k=n$.  
\end{enumerate} 
\end{proof}
As a special case of Lemma \ref{lp}, we have:
\begin{cor}\label{decomposition}
Let $\left(a_1,\ldots,a_k\right)$ be a pairwise distinct sequence of positive natural numbers. We
suppose further that the $p-$adic decompositions of two distinct numbers of the sequence have no
common $p-$power. Then if the equality $\som{i=1}{q}{p^{l_i}}=\som{j=1}{k}{a_j}$ holds, there exists
a partition $S_1,\ldots,S_k$ of $\left\{1,\ldots,q\right\}$ such that for all $i=1,\ldots,k$ we have
$\som{j\in S_i}{}{p^{l_j}}=a_i$.
\end{cor}

In order to prove Lemma \ref{prin2}, we need to determine the images in  $S^d(F(1)\otimes
V^{\sharp})$ of the element $$\omega_\alpha = \imath_{\lambda_1} \otimes \mathrm{P}_0^{\delta}
\imath_{\lambda_2} \otimes  \mathrm{P}_0^{2\delta} \imath_{\lambda_3} \otimes \cdots \otimes
\mathrm{P}_0^{(t-1)\delta} \imath_{\lambda_t}.$$ Because the degree of $\omega_{\alpha}$ is equal
to that of its image, we obtain an equality:
$$
\lambda_1 + p^{\delta}\lambda_2 +p^{2\delta}\lambda_3+\cdots +p^{(t-1)\delta}\lambda_t
=p^{l_1}+\cdots+p^{l_d}
.$$
The following lemma supplies a first determination of $l_i$ basing on this equality.
\begin{lem}\label{comb}
Let $\lambda=\left(\lambda_1,\ldots,\lambda_t\right)$ be a sequence of natural numbers such that the
sum is equal to $d$ and $\delta$ be a number such that $\delta>\max\left\{\lambda_i\right\}$ and
that the  $p-$adic length of $p^{\delta}-t$ is strictly greater than $d$ for $t\leq
dp^{\lambda_i}-\lambda_i$. We suppose furthermore that the following identity holds:
$$
\lambda_1 + p^{\delta}\lambda_2 +p^{2\delta}\lambda_3+\cdots +p^{(t-1)\delta}\lambda_t
=p^{l_1}+\cdots+p^{l_d},
$$
$l=\left(l_1,\ldots,l_d\right)$ denotes an ascending sequence. Then there exists a unique partition of $\{1,\ldots, d\}$ into $t$ subsets $E_i$ such that:
\begin{itemize}
\item For all $i$ we have 
$$
p^{(i-1)\delta} \lambda_i = \sum_{h \in E_i} p^{l_h};
$$
\item Each subset $E_i$ is a sequence of successive natural numbers $\{k,\ldots,k+r\}$;
\item $\mathrm{card}(E_1) \leq \lambda_1$.
\end{itemize}
\end{lem}
\begin{proof}
The condition $\delta>\max\left\{\lambda_i\right\}$ allows to apply Lemma \ref{decomposition}.

Let $\tau$ be the first index such that $p^{l_\tau} \leq \lambda_1 < p^{l_{\tau +1}}$. We show that
$\som{i=1}{\tau}{p^{l_i}}=\lambda_1$. 
The equality 
$$\left(\som{i=1}{\tau}{p^{l_i}}\right)-\lambda_1=p^{\delta}\left(\lambda_2
+p^{\delta}\lambda_3+\cdots +p^{(t-2)\delta}\lambda_t \right) - \som{i=\tau+1}{d}{p^{l_i}}$$
guarantees that $\left(\som{i=1}{\tau}{p^{l_i}}\right)-\lambda_1$ is divisible by $p^{\lambda_1}$.
It follows that $\left(\som{i=1}{\tau}{p^{l_i}}\right)\geq\lambda_1$. If the inequality is strict,
then $$\left(\som{i=1}{\tau}{p^{l_i}}\right)-\lambda_1\leq dp^{\lambda_1}-\lambda_1.$$ Hence the
$p-$adic length of the sum $$p^{\delta}\left(\lambda_2 +p^{\delta}\lambda_3+\cdots
+p^{(t-2)\delta}\lambda_t \right) - \left(\som{i=1}{\tau}{p^{l_i}}\right)+\lambda_1$$ is greater
than $d+1$. This contradicts the fact that $d$ is the upper bound of the $p-$adic length of
$\som{i=\tau+1}{d}{p^{l_i}}$. Hence $\som{i=1}{\tau}{p^{l_i}}=\lambda_1$. We can then choose
$E_1=\left\{1,2,\ldots,\tau \right\}$.
The rest of the lemma can be proved in the same manner by using Corollary \ref{decomposition}.

It follows from the equalities
$$
p^{(i-1)\delta} \lambda_i = \sum_{h \in E_i} p^{l_h}
$$
that $\mathrm{card}(E_i)\leq p^{(i-1)\delta} \lambda_i$ and in particular, $\mathrm{card}(E_1)\leq
\lambda_1$.
\end{proof}
It remains to prove that $\mathrm{card}(E_i)=\lambda_i$ to complete the proof of Lemma \ref{prin2}. Since $$\sum_{i}^{}\mathrm{card}(E_{i})=\sum_{i}^{}\lambda_{i},$$ it suffices to show that $\mathrm{card}(E_i)\leq\lambda_i$ for all $i$. Unfortunately, combinatorial arguments are not enough to reach the conclusion. We need to make use of the action of Steenrod algebra in order to realize these inequalities.
\subsection{Proof}\label{milnorop}
The proof of the key lemma now goes as follows.
\begin{proof}[Proof of Lemma \ref{prin2}]
The image $\omega_\alpha$ is a sum of elements of the type
$\prod_{i=1}^{d}u^{p^{l_{i}}}\otimes v_{i}$
 in $S^d(F(1) \otimes V^{\sharp})$. Therefore, the following equality holds:
$$
\lambda_1+p^{\delta}\lambda_2 +p^{2\delta}\lambda_3+\cdots +p^{(t-1)\delta}\lambda_t
=p^{l_1}+\cdots+p^{l_d}
$$

We now show that in $\prod_{i=1}^{d}u^{p^{l_{i}}}\otimes v_{i}$, $u$ appears  $\lambda_1$ times,
$u^{p^{\delta}}$ appears  $\lambda_2$ times and so on, $u^{p^{(i-1)\delta}}$ appears $\lambda_{i}$
times. Lemma \ref{comb} implies that $u$ appears at most
$\lambda_1$ times. By showing the same result for all $i$ we get $\mathrm{card}( E_i)=\lambda_i$ 
and the proof is completed.

It follows from Proposition \ref{Stee} that 
$$
m_{\lambda_{1}}(t\delta)\left(\omega_{\left(0,\delta,2\delta,\ldots,(t-1)\delta\right)}\right)=\omega_{(t{\delta},\delta, \ldots, (t-1)\delta)}.
$$
On the other hands
$$
\omega_{(t{\delta},\delta, \ldots, (t-1)\delta)}=\mathrm{P}_0^{\delta}\omega_{((t-1)\delta,0,\delta,
\ldots , (t-2)\delta)}.
$$
We denote by $k_1$ the index such that $\som{i=1}{k_1}{p^{l_i}}=\lambda_1$. Because
$$\som{i=k_1+1}{d}{p^{l_i}}=p^{\delta}\left(\lambda_2 +p^{\delta}\lambda_3+\cdots
+p^{(t-2)\delta}\lambda_t\right)$$
we must have $p^{l_{k_1+1}}>\lambda_1$. Otherwise, the $p-$adic length of $p^{\delta}-p^{l_{k_1+1}}$
is greater than $d$. Follow Proposition \ref{Stee}, we obtain:
\begin{align*}
m_{\lambda_{1}}(t\delta)\left(\prod_{i=1}^{d}u^{p^{l_{i}}}\otimes
v_{i}\right)&=\mathrm{P}_{0}^{t\delta}\left(\prod_{i=1}^{k_{1}}u^{p^{l_{i}}}\otimes
v_{i}\right)\cdot \prod_{i=1+k_{1}}^{d}u^{p^{l_{i}}}\otimes v_{i}
\end{align*}
Because every morphism in $\ho{\u}{F(\lambda)}{S^{d,V}(F(1))}$ is $\ap-$linear, it follows from
\begin{align*}
m_{\lambda_{1}}(t\delta)\left(\omega_{\left(0,\delta,2\delta,\ldots,(t-1)\delta\right)
}\right)&=\mathrm{P}_0^{\delta}\omega_{((t-1)\delta,0,\delta,
\ldots , (t-2)\delta)}
\end{align*}
that $l_{i}\geq \delta$ for all $i>k_1$. Moreover 
$p^{\delta}\lambda_2=\som{i\in E_2}{}{p^{l_i}}$
hence the equality $\lambda_2=\som{i\in E_2}{}{p^{l_i-\delta}}$ holds. 
It means that $\mathrm{Card}(E_2)\leq \lambda_2$. Similarly, by considering the elements
$$m_{\lambda_{1}+\lambda_2+\ldots+\lambda_i}(t\delta)\left(\omega_{\left(0,\delta,2\delta,\ldots,
(t-1)\delta\right)}\right)$$ we obtain the inequalities $\mathrm{Card}(E_{i+1})\leq \lambda_{i+1}$.
This concludes the lemma.
\end{proof}

\bibliographystyle{alpha}
\bibliography{Thuvien}
\end{document}